\documentclass[12pt, reqno]{amsart}
\usepackage{amsmath, amsthm, amscd, amsfonts, amssymb, graphicx, color}
\usepackage[bookmarksnumbered, colorlinks, plainpages]{hyperref}
\hypersetup{colorlinks=true,linkcolor=red, anchorcolor=green,
citecolor=cyan, urlcolor=red, filecolor=magenta, pdftoolbar=true}
\usepackage{cite}
\newtheorem{theorem}{Theorem}[section]
\newtheorem{lemma}[theorem]{Lemma}

\newtheorem{corollary}[theorem]{Corollary}
\theoremstyle{definition}

\theoremstyle{remark}
\newtheorem{remark}[theorem]{Remark}
\numberwithin{equation}{section}

\begin{document}

\setcounter{page}{1}

\title[2-Local derivations
on   matrix algebras]{2-Local derivations on   matrix algebras and
algebras of measurable operators}

\author[Sh. A. Ayupov, K. K. Kudaybergenov  \MakeLowercase{and}  A.K. Alauadinov]{Shavkat Ayupov$^1$, Karimbergen Kudaybergenov$^2$
\MakeLowercase{and} \\ Amir Alauadinov$^3$}
\address{$^1$Institute of
 Mathematics,  National University of
Uzbekistan,  Dormon yoli 29, 100125  Tashkent,   Uzbekistan}
\email{\textcolor[rgb]{0.00,0.00,0.84}{sh$_{-}$ayupov@mail.ru}}

\address{$^2$Department of Mathematics, Karakalpak State University, Ch. Abdirov 1, Nukus 230113, Uzbekistan}
\email{\textcolor[rgb]{0.00,0.00,0.84}{karim2006@mail.ru}}

\address{$^3$Department of Mathematics, Karakalpak State University, Ch. Abdirov 1, Nukus 230113, Uzbekistan}
\email{\textcolor[rgb]{0.00,0.00,0.84}{amir\(_{-}\)t85@mail.ru}}



\date{\today}
\maketitle

\begin{abstract} Let \(\mathcal{A}\) be a unital Banach  algebra
such that any Jordan derivation from  \(\mathcal{A}\) into any
\(\mathcal{A}\)-bimodule \(\mathcal{M}\) is a derivation. We prove
that any 2-local derivation from the algebra $M_n(\mathcal{A})$
into $M_n(\mathcal{M})$ $(n\geq 3)$ is a derivation. We apply this
result to  show that any 2-local derivation on the algebra of
locally measurable operators affiliated  with  a von Neumann algebra
without direct abelian summands  is a derivation.
\end{abstract}

\section{Introduction}

Let \(\mathcal{A}\) be an associative algebra over \(\mathbb{C}\)
the field of complex numbers  and let  \(\mathcal{M}\) be an
\(\mathcal{A}\)-bimodule. A linear map \(D\) from \(\mathcal{A}\)
to \(\mathcal{M}\) is called a \textit{derivation} if \(D(xy) =
D(x)y + xD(y)\) for all \(x, y\in \mathcal{A}.\) If it satisfies
a weaker condition  \(D(x^2) = D(x)x + xD(x)\) for every \(x\in
\mathcal{A}\) then it is called a \textit{Jordan derivation}.  It
is easy to verify that each element \(a \in \mathcal{M}\)
implements a derivation \(D_a\) from \(\mathcal{A}\) into
\(\mathcal{M}\) by \(D_a(x) = ax -xa,\, x \in \mathcal{A}.\) Such
derivations \(D_a\) are called \textit{inner derivations}.

In 1990, Kadison \cite{Kad}  and Larson and Sourour \cite{Lar} independently
introduced the concept of local derivation. A linear
map \(\Delta : \mathcal{A} \to \mathcal{M}\) is called a
\textit{local derivation} if  for every \(x\in\mathcal{A}\) there
exists a derivation \(D_x\)  (depending on \(x\)) such that
\(\Delta(x) = D_x(x).\) It would be interesting to consider under
which conditions local derivations automatically become
derivations. Many partial results have been done in this problem.
In \cite{Kad} Kadison shows that every norm-continuous local
derivation from a von Neumann algebra \(M\) into a dual
\(M\)-bimodule is a derivation. In~\cite{Joh} Johnson extends
Kadison's result and proves every local derivation from a
$C^{\ast}$-algebra $\mathcal{A}$ into any Banach
$\mathcal{A}$-bimodule is a derivation.

Similar problems for local derivations on algebras of measurable
operators \(S(M)\) and locally measurable operators \(LS(M),\)
affiliated with a von Neumann algebra \(M,\) have been considered
in \cite{AK2016} and \cite{HLLM}. Namely, it was proved that if
\(M\) is a von Neumann algebra without abelian direct summand then
every local derivation on \(LS(M)\) is a derivation. Moreover, for
abelian von Neumann algebras \(M\) necessary and sufficient
condition are given in \cite{AKA}  for \(S(M)=LS(M)\) to admit
local derivations which are not derivations (see for details the
survey \cite[Section 5]{AK2016}).

In 1997, \v{S}emrl \cite{Semrl97}  initiated the study of so-called  2-local derivations
and 2-local automorphisms on algebras. Namely, he  described such maps  on the algebra \(B(H)\) of all
bounded linear operators on an infinite dimensional separable Hilbert space  \(H\).

 In the above notations,  map \(\Delta : \mathcal{A} \to
\mathcal{M}\) (not necessarily linear) is called a \textit{2-local
derivation} if, for every \(x,y \in  \mathcal{A},\) there exists a
derivation \(D_{x,y} : \mathcal{A} \to \mathcal{M}\) such that
\(D_{x,y} (x) = \Delta(x)\) and \(D_{x,y}(y) = \Delta(y).\)

 Afterwards  local derivations and 2-local derivations have
been investigated by many authors on different algebras and many
results have been obtained in \cite{Ali,  AyuKuday2014, AK2016JP,
AKA,  Kad, KimKim04,  Semrl97}.

Recall that an algebra \(\mathcal{A}\) is called a regular (in the
sense of von Neumann) if for each \(a\in \mathcal{A}\) there
exists \(b\in\mathcal{A}\) such that \(a = aba.\) Let
\(M_n(\mathcal{A})\) be the algebra of all \(n \times n\) matrices
over a unital commutative regular algebra \(\mathcal{A}.\) In
\cite{AKA}, we prove that every 2-local derivation on
\(M_n(\mathcal{A}),\) \(n \geq 2,\) is a derivation.   We applied
this result to a description of 2-local derivations on the
algebras of measurable operators \(S(M)\) and locally measurable
operators \(LS(M)\) affiliated with a type I finite von Neumann
algebra \(M\). Further this result was extended to type
I\(_\infty\) von Neumann algebras: it was  proved that in this case 
every 2-local derivations on the algebra of locally measurable
operators is a derivation (see \cite[Theorem 6,7]{AK2016}).
Moreover in \cite{AKA} we also gave necessary and sufficient
conditions for a commutative regular algebra, in particular for
the algebra \(S(M)\) of measurable operators affiliated with an
abelian von Neumann algebra \(M\), to admit 2-local derivations
which are not derivations. In \cite{AK2016JP} we considered  a
unital semi-prime Banach algebra $\mathcal{A}$ with the inner
derivation property and proved that any 2-local derivation on the
algebra $M_{2^n}(\mathcal{A}),$ $n\geq 2,$ is a derivation. We
have applied this result to $AW^\ast$-algebras and proved that any
2-local derivation on an arbitrary $AW^\ast$-algebra is a
derivation. In \cite{HLQ}, W.~Huang, J.~Li and W.~Qian, have
characterized derivations and 2-local derivations from
\(M_n(\mathcal{A})\) into \(M_n(\mathcal{M}), n \geq 2,\) where
\(\mathcal{A}\) is a unital algebra over \(\mathbb{C}\) and
\(\mathcal{M}\) is a unital \(\mathcal{A}\)-bimodule.  They
considered   a unital Banach algebra such that any Jordan
derivation from the  algebra \(\mathcal{A}\) into any
\(\mathcal{A}\)-bimodule \(\mathcal{M}\) is an inner  derivation
and proved that any 2-local derivation from the algebra
$M_n(\mathcal{A})$ into $M_n(\mathcal{M})$ $(n\geq 3)$ is a
derivation, when \(\mathcal{A}\) is  commutative and commutes with
\(\mathcal{M}.\)

In the present paper we shall  consider  matrix
algebras over unital (non commutative in general)  Banach algebras and describe
 2-local derivations from \(M_n(\mathcal{A})\) into
\(M_n(\mathcal{M})\),  where  \(\mathcal{A}\) is  a unital
Banach algebra  such that any Jordan derivation
from the  algebra \(\mathcal{A}\) into any \(\mathcal{A}\)-bimodule
\(\mathcal{M}\) is a derivation. The main result of Section 2 asserts  that
under the above conditions every  2-local
derivation from the algebra $M_n(\mathcal{A})$ into
$M_n(\mathcal{M})$ $(n\geq 3)$ is a derivation.

In Section 3, we apply the main result of the previous section to
algebras of locally measurable operators affiliated with  von
Neumann algebras. Namely, we  extend all above mentioned results
from \cite{AK2016JP,  AK2016, AKA, HLQ} and prove that for an
arbitrary  von Neumann algebra \(M\)  without  abelian direct
summands
 every  2-local derivation on each  subalgebra
\(\mathcal{A}\) of the  algebra
 \(LS(M)\),  such that  \(M\subseteq\mathcal{A},\) is a derivation.
A similar result for local derivation is obtained in \cite[Theorem
1]{HLLM} (see also \cite[Theorem 5.5]{AK2016}).
\medskip

\section{2-local derivations on matrix algebras}

\medskip

If  $\Delta :\mathcal{A}\rightarrow \mathcal{M}$ is  a 2-local
derivation, then from the definition it easily follows that
$\Delta$ is homogenous. At the same time,
\begin{equation*}\label{joor}
 \Delta(x^2)=\Delta(x)x+x\Delta(x)
\end{equation*}
for each $x\in \mathcal{A}.$ This means that additive (and hence, linear) 2-local
derivation is a Jordan derivation.

In \cite{Bre2005} Bre\v{s}ar suggested various  conditions
on an algebra \(\mathcal{A}\) under which any Jordan derivation
from \(\mathcal{A}\) into any \(\mathcal{A}\)-bimodule
\(\mathcal{M}\) is a derivation.

In the present paper we shall consider algebras with the following
property:

\textbf{(J)}: \textit{any Jordan derivation from the algebra
\(\mathcal{A}\) into any \(\mathcal{A}\)-bimodule \(\mathcal{M}\)
is a derivation.}

Therefore, in the  case of  algebras with the property \textbf{(J)} in
order to prove that a 2-local derivation $\Delta :
\mathcal{A}\rightarrow \mathcal{M}$ is a derivation it is
sufficient to prove that $\Delta: \mathcal{A} \rightarrow
\mathcal{M}$ is additive.

Throughout this paper, \(\mathcal{A}\) is a unital Banach algebra
over \(\mathbb{C},\) \(\mathcal{M}\) is an
\(\mathcal{A}\)-bimodule with \(\mathbf{1} x = x \mathbf{1} = x\)
for all \(x\in \mathcal{M},\) where \(\mathbf{1}\) is the unit
element  of \(\mathcal{A}.\)

The following theorem is the main result of this section.

\begin{theorem}\label{mainlocal}
Let $\mathcal{A}$ be  a unital  Banach algebra with the property
\textbf{(J)}, $\mathcal{M}$ be a unital $\mathcal{A}$-bimodule and
let $M_{n}(\mathcal{A})$ be the algebra of all $n \times
n$-matrices over $\mathcal{A},$ where \(n\geq 3.\) Then any
2-local derivation $\Delta$ from $M_{n}(\mathcal{A})$ into
$M_{n}(\mathcal{M})$ is a derivation.
\end{theorem}

The proof  of Theorem~\ref{mainlocal} consists of two steps. In
the  first step  we shall show additivity of  $\Delta$ on the
subalgebra of diagonal matrices from $M_{n}(\mathcal{A}).$

Let $\{e_{i,j}\}_{i,j=1}^n$ be  the system of matrix units in
$M_n(\mathcal{A}).$ For $x \in M_n(\mathcal{A})$ by $x_{i,j}$ we
denote the $(i, j)$-entry of $x,$ where $1 \leq i, j \leq n.$ We
shall, if necessary, identify this element with the matrix from
$M_n(\mathcal{A})$ whose $(i,j)$-entry is $x_{i,j},$ other entries
are zero, i.e. $x_{i,j}=e_{i,i}xe_{j,j}.$

Each element \(x\in  M_n(\mathcal{A})\)  has the form
\[
x =\sum\limits_{i , j = 1}^n  x_{ij} e_{ij},\,\,  x_{ij}\in
\mathcal{A}, i , j \in  \overline{1, n}. \]

Let \(\delta : \mathcal{A} \to \mathcal{M}\) be a derivation.
Setting
\begin{equation}\label{cender}
\overline{\delta}(x) =\sum\limits_{i , j = 1}^n  \delta(x_{ij})
e_{ij},\,\, x_{ij}\in \mathcal{A}, i , j \in  \overline{1, n}
\end{equation}
we obtain a well-defined linear operator \(\overline{\delta}\)
from \(M_n(\mathcal{A})\)  into \(M_n(\mathcal{M}).\) Moreover
\(\overline{\delta}\) is a derivation from \(M_n(\mathcal{A})\)
into \(M_n(\mathcal{M}).\)

It is known \cite[Theorem 2.1]{HLQ} that every derivation \(D\)
from \(M_n(\mathcal{A})\) into \(M_n(\mathcal{M})\) can be
represented as a sum
\begin{equation}\label{decompos}
D = ad(a) + \overline{\delta},
\end{equation}
where \(\textrm{ad}(a)\) is an inner derivation implemented by an
element \(a \in  M_n(\mathcal{M}),\) while \(\overline{\delta}\)
is the derivation of the form \eqref{cender}  generated by a
derivation \(\delta\) from \(\mathcal{A}\) into \(\mathcal{M}.\)

Consider the following two matrices:
\begin{equation}\label{vv}
 u=\sum\limits_{i=1}^n \frac{1}{2^i}e_{i,i},\,
v=\sum\limits_{i=2}^n e_{i-1,i}.
\end{equation}

It is easy to see that an element $x \in M_n(\mathcal{M})$
commutes with $u$  if and only if it is diagonal, and if an
element  $a \in M_n(\mathcal{M})$ commutes with $v,$ then $a$ is
of the form
\begin{equation}\label{trian}
a=\left( \begin{array}{ccccccc}
a_1   & a_2   & a_3 & . & \ldots & a_n \\
0     & a_1   & a_2 & . & \ldots & a_{n-1}\\
0     & 0     & a_1 & . & \ldots & a_{n-2}\\
\vdots& \vdots& \vdots & \vdots  &\vdots & \vdots\\
0 & 0 &\ldots & . & a_1 & a_2\\
0 & 0 &\ldots & .& 0 & a_1
\end{array} \right).
\end{equation}

A  result, similar to the following one,  was proved in
\cite[Lemma 4.4]{AKA} for matrix algebras over commutative regular
algebras.

Further in Lemmata~\ref{lemmatwo}--\ref{lemmafive} we assume that
$n\geq 2.$

\begin{lemma}\label{lemmatwo}
For every $2$-local derivation $\Delta$ from $M_n(\mathcal{A})$
into $M_n(\mathcal{M})$ there exists a derivation  $D$   such that
$\Delta|_{\mbox{sp}\{e_{i,j}\}_{i,j=1}^n}=D|_{\mbox{sp}\{e_{i,j}\}_{i,j=1}^n},$
where $\mbox{sp}\{e_{i,j}\}_{i,j=1}^n$ is the linear span of the
set $\{e_{i,j}\}_{i,j=1}^n.$
\end{lemma}

\begin{proof}  Take a derivation \(D\) from  \(M_n(\mathcal{A})\) into $M_n(\mathcal{M})$ such that
$$
\Delta(u)=D(u),\, \Delta(v)=D(v),
$$
where $u, v$ are the elements from~\eqref{vv}. Replacing $\Delta$
by $\Delta-D$, if necessary, we can assume that
$\Delta(u)=\Delta(v)=0.$

Let $i, j\in \overline{1, n}.$ Take a derivation \(D =
\textrm{ad}(h) + \overline{\delta}\) of the form \eqref{decompos}
such that
$$
\Delta(e_{i,j})=[h,
e_{i,j}]+\overline{\delta}(e_{ij}),\,\Delta(u)=[h,
u]+\overline{\delta}(u).
$$
Since $\Delta(u)=0$ and \(\overline{\delta}(u)=0,\)   it follows
that \([h, u]=0,\)  and therefore \(h\) has a diagonal form, i.e.
$h=\sum\limits_{s=1}^n h_{s}e_{s,s},\, h_s\in \mathcal{A},\, s\in
\overline{1, n}.$

In the same way, but starting with the element $v$ instead of $u$,
we obtain
$$
\Delta(e_{i,j})=be_{i,j}-e_{i,j}b,
$$
where  $b$  has the  form~\eqref{trian}, depending on $e_{i,j}.$
So
$$
\Delta(e_{i,j})=he_{i,j}-e_{i,j} h=b e_{i,j}-e_{i,j} b.
$$
Since
$$
he_{i,j} - e_{i,j} h= (h_{i}-h_{j})e_{i,j}
$$
 and
 $$
[b e_{i,j}- e_{i,j} b]_{i,j}=0,
$$
 it follows that $\Delta(e_{i,j})=0.$

Now let us take a matrix $x=\sum\limits_{i,j=1}^n
\lambda_{i,j}e_{i,j}\in M_n(\mathbb{C}).$ Then
\begin{eqnarray*}
e_{i,j}\Delta(x)e_{i,j} & = &
e_{i,j} D_{e_{i,j}, x} (x) e_{i,j} =\\
& = & D_{e_{i,j}, x} (e_{i,j} x e_{i,j})-D_{e_{i,j}, x}
(e_{i,j})x e_{i,j}-e_{i,j} x D_{e_{i,j}, x} (e_{i,j})=\\
& = & D_{e_{i,j}, x} (\lambda_{j,i}  e_{i,j})-\Delta(e_{i,j})x e_{i,j}-e_{i,j} x \Delta(e_{i,j})=\\
&=& \lambda_{j,i}  D_{e_{i,j}, x} (e_{i,j}) -0-0=\lambda_{j,i}
\Delta(e_{i,j}) =0,
\end{eqnarray*}
i.e. $ e_{i,j}\Delta(x)e_{i,j}=0$ for all $i,j\in \overline{1,n}.$
This means that $\Delta(x)=0.$ The proof is complete. \end{proof}

Further in Lemmata~\ref{three}--\ref{adj} we assume that $\Delta$
is a 2-local derivation from  $M_n(\mathcal{A})$ into
$M_n(\mathcal{M})$ such that
$\Delta|_{\mbox{sp}\{e_{i,j}\}_{i,j=1}^n}= 0.$

Let $\Delta_{i,j}$ be the restriction of $\Delta$ onto
$\mathcal{A}_{i,j}=e_{i,i}M_n(\mathcal{A})e_{j,j},$ where $1 \leq
i, j \leq n.$

\begin{lemma}\label{three}
$\Delta_{i,j}$ maps $\mathcal{A}_{i,j}$ into itself.
\end{lemma}

\begin{proof}  Let us show that
\begin{equation}\label{compo}
\Delta_{i,j}(x) =e_{i,i}  \Delta(x) e_{j,j}
\end{equation}
for all $x\in \mathcal{A}_{i,j}.$

Take $x=x_{i,j}\in \mathcal{A}_{i,j},$  and consider a derivation \(D =
\textrm{ad}(h) + \overline{\delta}\) of the form \eqref{decompos}
such that
$$
\Delta(x)=[h, x]+\overline{\delta}(x),\,\Delta(u)=[h,
u]+\overline{\delta}(u),
$$
where \(u\) is the element from \eqref{vv}. Since $\Delta(u)=0$
and \(\overline{\delta}(u)=0,\) it follows that \([h, u]=0,\)  and
therefore \(h\) has a diagonal form. Then \(\Delta(x) =
(h_{i}-h_{j})e_{i j} +\delta(x_{i j})e_{i j}.\) This means that
\(\Delta(x) \in \mathcal{A}_{i,j}.\) The proof is complete.
\end{proof}

\begin{lemma}\label{lemmafour} Let
$x=\sum\limits_{i=1}^n x_{i,i}$ be a diagonal matrix. Then
\begin{equation}\label{kkkk}
e_{k,k}\Delta(x)e_{k,k}=\Delta(x_{k,k})
\end{equation} for all $k\in
\overline{1,n}.$
\end{lemma}

\begin{proof}  Take a derivation \(D=\textrm{ad}(a)+\overline{\delta}\) of the form
\eqref{decompos}  such that
\begin{center}
\(\Delta(x)=[a, x]+\overline{\delta}(x)\) and
\(\Delta(x_{k,k})=[a, x_{k,k}]+\overline{\delta}(x_{kk}).\)
\end{center}
Using equality \eqref{compo}, we obtain that
\begin{eqnarray*}
\Delta(x_{k,k}) & = & e_{k,k}\Delta(x_{k,k})e_{k,k}=e_{k,k}[a,
x_{k,k}] e_{k,k} +e_{k,k}\overline{\delta}(x_{k,k}) e_{k,k}=\\
& = & [a_{k,k}, x_{k,k}]+\delta(x_{k,k}).
\end{eqnarray*}
Since  $x$ is a diagonal matrix, we get
\begin{eqnarray*}
e_{k,k}\Delta(x)e_{k,k} & = & e_{k,k}[a, x] e_{k,k}
+e_{kk}\overline{\delta}(x)e_{k,k} = [a_{k,k},
x_{k,k}]+\delta(x_{k,k}).
\end{eqnarray*}
Thus $e_{k,k}\Delta(x)e_{k,k}=\Delta(x_{k,k}).$ The proof is
complete.
\end{proof}

\begin{lemma}\label{lemmafive} Let  $x=x_{i,i}\in
\mathcal{A}_{i,i}.$ Then
\begin{equation}\label{jiji}
e_{j,i}\Delta(x)e_{i,j}=\Delta(e_{j,i}xe_{i,j})
\end{equation}
for every \(j\in \{1, \cdots, n\}.\)
\end{lemma}

\begin{proof}  For $i=j$ we have already proved (see
Lemma~\ref{lemmafour}).

Suppose that $i\neq j.$  For an arbitrary element $x=x_{i,i}\in
\mathcal{A}_{i,i}$ , consider  $y=x+e_{j,i}xe_{i,j}\in
\mathcal{A}_{i,i}+\mathcal{A}_{j,j}.$ Take a derivation
\(D=\textrm{ad}(a)+\overline{\delta}\)   such that
\begin{center}
$\Delta(y)=[a, y]+\overline{\delta}(y)$ and $\Delta(v)=[a,
v]+\overline{\delta}(v),$
\end{center}
where $v$ is the element from~\eqref{vv}. Since $\Delta(v)=0$ and
\(\overline{\delta}(v)=0,\) it follows that $a$ has the
form~\eqref{trian}. By Lemma~\ref{lemmafour} we obtain that
\begin{eqnarray*}
e_{j,i}\Delta(x)e_{i,j} & = &
e_{j,i}e_{i,i}\Delta(y)e_{i,i}e_{i,j} = e_{j,i}[a,
y]e_{i,j}+e_{j,i}\overline{\delta}(y)e_{i,j}=\\
& = & \left([a_1,
x]+\delta(x)\right)e_{j,j},\\
 \Delta(e_{j,i}xe_{i,j}) & = &
e_{j,j}\Delta(y)e_{j,j}=e_{j,j}[a, y]e_{j,j}+e_{j,j}\overline{\delta}(y)e_{j,j}=\\
& = &
e_{j,j}[a,x+e_{j,i}xe_{i,j}]e_{j,j}+e_{j,j}\delta(x)e_{j,j}=\left([a_1,
x]+\delta(x)\right)e_{j,j}.
\end{eqnarray*}
The proof is complete. \end{proof}

Further in Lemmata~\ref{six}--\ref{zz} we assume that $n\geq 3.$

\begin{lemma}\label{six}
 $\Delta_{i,i}$ is additive for all $i\in \overline{1,n}.$
\end{lemma}

\begin{proof}  Let $i\in \overline{1,n}.$ Since $n\geq 3,$ we
can take different numbers $k, s$ such that \linebreak
$(k-i)(s-i)\neq 0.$

For arbitrary $x, y\in \mathcal{A}_{i,i}$  consider the diagonal
element $z\in
\mathcal{A}_{i,i}+\mathcal{A}_{k,k}+\mathcal{A}_{s,s}$ such that
\(z_{ i,i} = x+y,\, z_{k, k} = x,\, z_{s, s} = y.\) Take a
derivation \(D=\textrm{ad}(a)+\overline{\delta}\) such that
\begin{center}
$\Delta(z)=[a, z]+\overline{\delta}(z)$ and $\Delta(v)=[a,
v]+\overline{\delta}(v),$
\end{center}
where $v$ is  the element from~\eqref{vv}. Since $\Delta(v)=0$ and
\(\overline{\delta}(v)=0,\) it follows that $a$ has the
form~\eqref{trian}. Using Lemmata~\ref{lemmafour} and
\ref{lemmafive} we obtain that
\begin{eqnarray*}
\Delta_{i,i}(x+y) &   \stackrel{\eqref{kkkk}}{=}  &
e_{i,i}\Delta(z)e_{i,i}=e_{i,i}[a,
z]e_{i,i}+e_{i,i}\overline{\delta}(z)e_{i,i}=\\
& = & \left([a_1, x+y]+\delta(x+y)\right)e_{i,i},\\
 \Delta_{i,i}(x) &  \stackrel{\eqref{jiji}}{=}  & e_{i,k}\Delta(e_{k,i}x
e_{i,k})e_{k,i}\stackrel{\eqref{kkkk}}{=} e_{i,k}e_{k,k}\Delta(z)e_{k,k}e_{k,i}=\\
&  =  & e_{i,k}[a,
z]e_{k,i}+e_{i,k}\overline{\delta}(z)e_{k,i}=\left([a_1, x]+\delta(x)\right)e_{i,i},\\
\Delta_{i,i}(y) &  \stackrel{\eqref{jiji}}{=} &
e_{i,s}\Delta(e_{s,i}y
e_{i,s})e_{s,i}\stackrel{\eqref{kkkk}}{=} e_{i,s}e_{s,s}\Delta(z)e_{s,s}e_{s,i}=\\
& = & e_{i,s}[a,
z]e_{s,i}+e_{i,s}\overline{\delta}(z)e_{s,i}=\left([a_1,
y]+\delta(y)\right)e_{i,i}.
\end{eqnarray*} Hence
$$
\Delta_{i,i}(x+y)=\Delta_{i,i}(x)+\Delta_{i,i}(y).
$$
The proof is complete. \end{proof}

As it was mentioned in the beginning of the section  any additive
2-local derivation  is a Jordan derivation. Since
$\mathcal{A}_{i,i}\cong \mathcal{A}$ has the property
\textbf{(J)}, Lemma~\ref{six} implies the following result.

\begin{lemma}\label{lemmaseven}
 $\Delta_{i,i}$ is a derivation for all $i\in \overline{1,n}.$
\end{lemma}

Denote by $\mathcal{D}_n(\mathcal{A})$ the set of all diagonal
matrices from $M_n(\mathcal{A}),$ i.e. the set of all matrices of
the following form
\[
x=\left( \begin{array}{cccccc}
x_1 & 0 & 0 & \ldots & 0 \\
0 & x_2 &  0 & \ldots & 0\\
\vdots& \vdots& \vdots &\vdots & \vdots\\
0 & 0 &\ldots & x_{n-1} & 0\\
0 & 0 &\ldots & 0 & x_n
\end{array} \right).
\]

Let us consider a derivation $\overline{\Delta_{1,1}}$ of the
form~\eqref{cender}. By Lemmata~\ref{lemmafour} and
\ref{lemmafive} we obtain that

\begin{lemma}\label{adj}
$\Delta|_{\mathcal{D}_n(\mathcal{A})}=\overline{\Delta_{1,1}}|_{\mathcal{D}_n(\mathcal{A})}$
and $\overline{\Delta_{1,1}}|_{\mbox{sp}\{e_{i,j}\}_{i,j=1}^n}=0.$
\end{lemma}

Now we are in position to pass to the second step of our proof. In
this step we show that if a 2-local derivation $\Delta$ satisfies
the following conditions
\begin{center}
\(\Delta|_{\mathcal{D}_n(\mathcal{A})}\equiv 0\) and
\(\Delta|_{\mbox{sp}\{e_{i,j}\}_{i,j=1}^n}\equiv 0,\)
\end{center}
then it is identically equal to zero.

Below in the five Lemmata we
shall consider 2-local derivations which satisfy the latter equalities.

We denote by $e$ the unit of the algebra $\mathcal{A}.$

\begin{lemma}\label{ss}
Let $x\in M_n(\mathcal{A}).$ Then $\Delta(x)_{k,k}=0$ for all
\(k\in \overline{1,n}.\)
\end{lemma}

\begin{proof}  Let $x\in M_n(\mathcal{A}),$ and fix \(k\in \overline{1,n}.\)
Since $\Delta$ is
homogeneous, we can assume that $\|x_{k,k}\|<1,$ where $\|\cdot\|$
is the norm on $\mathcal{A}.$ Take a diagonal element $y$ in
\(M_n(\mathcal{A})\) with \(y_{k,k}=e+x_{k,k}\) and \(y_{i,i}=0\)
otherwise.  Since $\|x_{k,k}\|<1,$ it follows that $e+x_{k,k}$ is
invertible in $\mathcal{A}.$ Take a derivation
\(D=\textrm{ad}(a)+\overline{\delta}\) of the form
\eqref{decompos}  such that
$$
\Delta(x)=[a,x]+\overline{\delta}(x),\, \Delta(y)=[a
,y]+\overline{\delta}(y).
$$
Since  $y\in \mathcal{D}_n(\mathcal{A})$ we have that
$0=\Delta(y)=[a ,y]+\overline{\delta}(y),$ and therefore
\begin{eqnarray*}
0 &
= & \Delta(y)_{k,k}=a_{k,k}(e+x_{k,k})-(e+x_{k,k})a_{k,k}+\delta(e+x_{k,k})=0,\\
0 &  = &\Delta(y)_{i,k}=a_{i,k}(e +x_{k,k})=0,\\
0 &  = & \Delta(y)_{k,i}=-(e +x_{k,k}) a_{k,i}=0
\end{eqnarray*}
for all \(i\neq k.\) Thus
$$
a_{k,k}x_{k,k}-x_{k,k}a_{k,k}+\delta(x_{k,k})=0
$$
and
$$
a_{i,k}=a_{k,i}=0
$$
for all \(i\neq k.\) The above equalities imply that
$$
\Delta(x)_{k,k}=a_{k,k}x_{k,k}-x_{k,k}a_{k,k}+\delta(x_{k,k})=\Delta(y)_{k,k}=0.
$$
The proof is complete. \end{proof}

\begin{lemma}\label{zerro}
Let   $x$ be a matrix with $x_{k,s}=e.$  Then $\Delta(x)_{k,s}=0.$
\end{lemma}

\begin{proof} We have
\begin{eqnarray*}
e_{s,k}\Delta(x)e_{s,k} & = &
e_{s,k}D_{e_{s,k}, x} (x) e_{s,k}  =\\
& = & D_{e_{s,k}, x} (e_{s,k} x e_{s,k})-D_{e_{s,k}, x}(e_{s,k})
x e_{s,k}-e_{s,k} x D_{e_{s,k}, x} (e_{s,k})=\\
& = & D_{e_{s,k}, x} (e_{s,k})-\Delta(e_{s,k})
x e_{s,k}-e_{s,k} x \Delta(e_{s,k})=\\
&=& \Delta(e_{s,k})-0-0=0.
\end{eqnarray*}
Thus
$$
e_{k,k}\Delta(x)e_{s,s}=e_{k,s}e_{s,k}\Delta(x)e_{s,k}e_{k,s}=0.
$$
This means that \(\Delta(x)_{k,s}=0.\) The proof is complete.
\end{proof}

\begin{lemma}\label{cc} Let \(k, s\) be numbers such that \(k\neq s\) and let
 $x$ be a matrix with $x_{k,s}=e.$ Then $\Delta(x)_{s,k}=0.$
\end{lemma}

\begin{proof} Take a diagonal
element \(y\) such that \(y_{k,k}=x_{s,k}\) and
\(y_{i,i}=\lambda_i e\) otherwise, where \(\lambda_i\,\, (i\neq
k)\) are distinct numbers with \(|\lambda_i|>\|x_{s,k}\|.\)  Take
a derivation \(D=\textrm{ad}(a)+\overline{\delta}\) such that
\begin{center}
$\Delta(x)=[a, x]+\overline{\delta}(x)$ and $\Delta(y)=[a,
y]+\overline{\delta}(y).$
\end{center}
Then
\begin{eqnarray*}
  0 & = & \Delta(y)_{ij} = \lambda_j a_{i,j}-\lambda_i a
_{i,j}=a_{i,j}(\lambda_j -\lambda_i),\, i\neq j,\, (i-k)(j-k)\neq 0,\\
 0 & = & \Delta(y)_{i,k} = a_{i,k}y_{k,k}-\lambda_i a_{i,k}
=a_{i,k}(x_{s, k}-\lambda_i),\, i\neq k,\\
0 & = &   \Delta(y)_{k,j} = a_{k,j}\lambda_{j}-y_{kk} a_{k j}
=(\lambda_j-x_{s, k})a_{k,j},\, j\neq k.
\end{eqnarray*}
Thus
 \(a_{i,j}=0\) for all \(i\neq j,\) i.e.  \(a\) is a diagonal
element. Since
\[
0=\Delta(x)_{ks}=a_{kk}-a_{ss},
\]
it follows that \(a_{k,k}=a_{s,s}.\) Finally,
\begin{eqnarray*}
\Delta(x)_{s,k} & = &
a_{s,s}x_{s,k}-x_{s,k}a_{k,k}+\delta(x_{s,k})=\\
& = &
a_{k,k}x_{s,k}-x_{s,k}a_{k,k}+\delta(y_{k,k})=\Delta(y)_{k,k}=0.
\end{eqnarray*}
The proof is complete. \end{proof}

\begin{lemma}\label{z} Let
$k\neq s$ and let \(x,\) \(y\) be matrices with \(x_{i,j}=y_{i,j}\)
for all \((i,j)\neq (s,k).\) Then
$\Delta(x)_{k,s}=\Delta(y)_{k,s}.$
\end{lemma}

\begin{proof} Take a derivation \(D=\textrm{ad}(a)+\overline{\delta}\)  such that
\begin{center}
$\Delta(x)=[a, x]+\overline{\delta}(x)$ and $\Delta(y)=[a,
y]+\overline{\delta}(y).$
\end{center}
Then
\begin{eqnarray*}
\Delta(x)_{k,s} & = & \sum\limits_{j=1}^n
\left(a_{k,j}x_{j,s}-x_{k,j}a_{j,s}\right)+\delta(x_{ks})=\\
& = & \sum\limits_{j=1}^n
\left(a_{k,j}y_{j,s}-y_{k,j}a_{j,s}\right)+\delta(y_{ks})=\Delta(y)_{k,s}.
\end{eqnarray*} The proof is complete. \end{proof}

\begin{lemma}\label{zz}
Let \(k\neq s.\) Then $\Delta(x)_{k,s}=0.$
\end{lemma}

\begin{proof} Take a matrix  \(y\) with \(y_{s,k}=e\) and
\(y_{i,j}=x_{i,j}\) otherwise. By Lemma~\ref{cc} we have that
\(\Delta(y)_{k,s}=0.\) Further Lemma~\ref{z} implies that
$$
\Delta(x)_{k,s}=\Delta(y)_{k,s}=0.
$$
The proof is complete. \end{proof}

Now we are in position to prove
Theorem~\ref{mainlocal}.

\textit{Proof of Theorem~\ref{mainlocal}}.  Let $\Delta$ be a
2-local derivation from  $M_{n}(\mathcal{A})$ into
$M_{n}(\mathcal{M}),$ where $n\geq 3.$ By Lemma~\ref{lemmatwo}
there exists a derivation $D$  such that
$\Delta|_{\mbox{sp}\{e_{i,j}\}_{i,j=1}^{n}}=D|_{\mbox{sp}\{e_{i,j}\}_{i,j=1}^{n}}.$
Consider a 2-local derivation \(\Theta=\Delta-D.\) Since $\Theta$
is equal to zero on $\mbox{sp}\{e_{i,j}\}_{i,j=1}^{n},$ by
Lemma~\ref{adj} we obtain that
$\Theta|_{\mathcal{D}_{n}(\mathcal{A})}=\overline{\Theta_{11}}|_{\mathcal{D}_{n}(\mathcal{A})},$
where \(\overline{\Theta_{11}}\) is the derivation defined
by~\eqref{cender}. As in Lemma~\ref{adj} we have that
\begin{center}
$\left(\Theta-\overline{\Theta_{1
1}}\right)|_{\mbox{sp}\{e_{i,j}\}_{i,j=1}^{n}}\equiv 0$ and
$\left(\Theta-\overline{\Theta_{1
1}}\right)|_{\mathcal{D}_{n}(\mathcal{A})}\equiv 0.$
\end{center}
Now for an arbitrary element \(x\in M_n(\mathcal{A}),\) by
Lemmata~\ref{ss} and \ref{zz} we obtain that
\(\left(\Theta-\overline{\Theta_{1 1}}\right)(x)_{k,s}=0\) for all
\(k,s.\) Thus \(\left(\Theta-\overline{\Theta_{1
1}}\right)(x)=0,\) i.e., \(\Theta=\overline{\Theta_{1 1}}.\) So,
\(\Delta=\overline{\Theta_{1 1}}+D\) is a derivation. The proof is
complete. $\Box$

\section{An application to 2-local derivations on algebras of locally measurable operators}

In this section we apply Theorem~\ref{mainlocal} to the
description of 2-local derivations on  the algebra of
locally  measurable operators affiliated with a von Neumann algebra
and on its subalgebras.

In \cite[Corollary 3.11]{Bre2005} it was proved that if  an associative algebra (ring)
\(\mathcal{A}\) contains a noncommutative simple subalgebra (subring)  \(\mathcal{A}_0\)
which contains the unit of \(\mathcal{A}\), then
every Jordan derivation from \(\mathcal{A}\)  into any
\(\mathcal{A}\)-bimodule is a derivation, i.e. \(\mathcal{A}\)
satisfies the property \textbf{(J)}. In particular, if   there
exists a subalgebra   \(\mathcal{A}_0\) of \(\mathcal{A}\) which is
isomorphic to \(M_n(\mathbb{C})\) (\(n\geq 2\))  and contains the
unit of \(\mathcal{A},\) then  \(\mathcal{A}\) has the
property \textbf{(J)}.

Let \(M\) be a von Neumann algebra and denote by \(S(M)\) the
algebra of all measurable operators and by \(LS(M)\) � the algebra
of all locally measurable operators affiliated with \(M\) (see for
example \cite{MA07, Seg}).

\begin{theorem}\label{kaplanal} Let $M$ be an arbitrary
von Neumann algebra without abelian direct summands and let
\(LS(M)\) be the algebra of all locally measurable operators affiliated with
 \(M.\) Then any 2-local derivation $\Delta$ from \(M\)
into $LS(M)$ is a derivation.
\end{theorem}

\begin{proof}  Let $z$ be a central projection in $M.$ Since $D(z)=0$
for an arbitrary derivation $D,$ it is clear that $\Delta(z)=0$
for any $2$-local derivation $\Delta$ from \(M\) into $LS(M).$
Take $x\in M$ and let $D$ be a derivation from \(M\) into $LS(M)$
such that $\Delta(z x)=D(z x), \Delta(x)=D(x)$. Then we have
$\Delta(z x)=D(z x)=D(z)x+zD(x)=z\Delta(x).$ This means that every
2-local derivation $\Delta$ maps $zM$ into $zLS(M)\cong LS(z M)$
for each central projection $z\in M.$ So, we may consider the
restriction of $\Delta$ onto $zM.$ Since an arbitrary von Neumann
algebra without abelian direct summands can be decomposed along a
central projection into the direct sum of von Neumann algebras
 of  type I$_n, n\geq2,$ type I$_\infty,$ type II and type III, we may consider these cases separately.

If  $M$ is a von Neumann algebra of  type I$_n,$ $n\geq 2,$
\cite[Corollary 3.12]{HLQ} implies that any 2-local derivation
from \(M\) into \(LS(M)\equiv S(M)\)  is a derivation.

Let  the von Neumann algebra $M$ have one of the types I$_\infty,$
II or III. Then  the halving Lemma \cite[Lemma 6.3.3]{KR} for type
I$_\infty$-algebras and \cite[Lemma 6.5.6]{KR} for type II or III
algebras, imply that the unit of the algebra $M$ can be
represented as a sum of mutually equivalent orthogonal projections
$e_1, e_2, e_3$ from $M.$ Then the map $x\mapsto
\sum\limits_{i,j=1}^3 e_ixe_j$ defines an  isomorphism between the
algebra $M$ and the matrix algebra $M_3(\mathcal{A}),$ where
$\mathcal{A}=e_{1,1}Me_{1,1}.$ Further, the algebra \(LS(M)\) is
isomorphic to the algebra \(M_3(LS(\mathcal{A})).\) Moreover, the
algebra \(\mathcal{A}\) has same type as the algebra \(M,\) and
therefore contains a subalgebra isomorphic to \(M_3(\mathbb{C}).\)
This means that the algebra \(\mathcal{A}\) satisfies the property
\textbf{(J)}. Therefore Theorem~\ref{mainlocal} implies that any
2-local derivation from  $M$ into \(LS(M)\) is a derivation. The
proof is complete.
\end{proof}

Taking into account that any derivation on an abelian von Neumann
algebra is trivial, Theorem~\ref{kaplanal} implies the following
result (cf. \cite[Theorem 2.1]{AyuKuday2014} and \cite[Theorem
3.1]{AK2016JP}).

\begin{corollary}\label{localcor}
Let $M$ be an arbitrary von Neumann algebra. Then any 2-local
derivation $\Delta$ on  $M$ is a derivation.
\end{corollary}

For each \(x \in  LS(M)\)  set \(s(x) = l(x) \vee  r(x),\) where
\(l(x)\) is the left and \(r(x)\) is the right support of \(x.\)

\begin{lemma}\label{zerro} Let \(\mathcal{B}\) be a subalgebra of \(LS(M)\)  such that \(M \subseteq
\mathcal{B}\) and let \(\Delta: \mathcal{B} \to LS(M)\) be a
2-local derivation such that \(\Delta|_M\equiv 0.\) Then
\(\Delta\equiv 0.\)
\end{lemma}

\begin{proof} Let us first take an  arbitrary element \(x\in  \mathcal{B} \cap S(M).\) Let \(|x| =
\int\limits_0^\infty \lambda\, de_\lambda\) be  the spectral
resolution of \(|x|.\) Since \(x\in S(M),\) it follows that
\(e_n^\perp\) is a finite projection for a sufficiently large
\(n.\)  Take a derivation \(D_{x, xe_n}\) such that
 \(\Delta(x)=D_{x, xe_n}(x)\) and \(\Delta(xe_n)=D_{x, xe_n}(xe_n),\) \(n\in \mathbb{N}.\)  Since \(x
 e_n\in M,\) it follows that \(\Delta(xe_n)=0\) for all  \(n\in \mathbb{N}.\) We have
 \begin{eqnarray*}
\Delta(x) & = & \Delta(x)-\Delta(xe_n)=D_{x, xe_n}(x)-D_{x, xe_n}(xe_n)=\\
& = & D_{x, xe_n}(x-xe_n)=D_{x, xe_n}(xe_n^\perp).
\end{eqnarray*}
Let \(\mathcal{D}\) be a dimension function on the lattice
\(P(M)\) of all projections from  \(M\) (see \cite{Seg}). Using
\cite[Lemma 4.3]{Ber2011} we obtain that
 \begin{eqnarray*}
\mathcal{D}(s(\Delta(x))) & = & \mathcal{D}(s(D_{x,
xe_n}(xe_n^\perp)))\leq 3\mathcal{D}(s(xe_n^\perp))=
3\mathcal{D}(l(xe_n^\perp)\vee r(xe_n^\perp))\leq\\ & \leq &
3\mathcal{D}\left(l(xe_n^\perp)\right)+3\mathcal{D}(r(xe_n^\perp))\leq6\mathcal{D}(e_n^\perp)\downarrow
0,
\end{eqnarray*}
and therefore \(\Delta(x)=0.\)

Now let take an element \(x\in \mathcal{B}.\) By the  definition of
locally measurable operator there exists a sequence \(\{z_n\}\) of
central projections in \(M\)  such that \(z_n \uparrow
\mathbf{1}\) and \(xz_n \in S(M)\) for all \(n \in \mathbb{N}\)
(see \cite{MA07}).  Taking into account the previous case we
obtain that
\begin{eqnarray*}
z_n\Delta(x) & = & z_n D_{x, z_n x}(x)= D_{x, z_n x}(z_n x)-D_{x, z_n x}(z_n)x=\\
& = & D_{x, z_n x}(z_nx)=\Delta(z_n x)=0,
\end{eqnarray*}
i.e., \(z_n \Delta(x)=0\) for all \(n \in  \mathbb{N}.\) Hence
\(\Delta(x)=0.\) The proof is complete.
\end{proof}

\begin{theorem}\label{localmeas}(cf. \cite[Theorem 5.5]{AK2016}).
Let $M$ be an arbitrary von Neumann algebra without abelian direct
summands and let \(\mathcal{B}\) be a subalgebra of \(LS(M)\) such
that \(M \subseteq \mathcal{B}.\) Then any 2-local derivation
$\Delta$ on  $\mathcal{B}$ is a derivation.
\end{theorem}

\begin{proof}
By Theorem~\ref{kaplanal} the restriction \(\Delta|_M\) of
\(\Delta,\) is a derivation from \(M\) into \(LS(M).\) By
\cite[Theorem 4.8]{Ber2011} the derivation \(\Delta|_M\)  can be
extended to a  derivation from \(\mathcal{B}\) into  \(LS(M),\)
which we denote by \(D.\) Since the 2-local derivation
\(\Delta-D\) is equal to zero on \(M,\) Lemma~\ref{zerro} implies
that \(\Delta\equiv D.\) The proof is complete.
\end{proof}

\begin{remark}
As it was mentioned in the introduction, the paper \cite{AKA}
gives necessary and sufficient conditions on a commutative regular
algebra to admit 2-local derivations which are not derivations. In
particular, for an arbitrary abelian von Neumann algebra \(M\) with a non
atomic lattice of projections \(P(M)\) the algebras \(S(M)\)  and \(LS(M)\) always admit
a 2-local derivation which is not a derivation.

A complete description  of derivations on the algebra $LS(M)$
for type I von Neumann algebras $M$ is given in   \cite[Section
3]{AK2016}). Moreover, for general von Neumann algebras every
derivation on the algebra $LS(M)$ is inner, provided that $M$ is a
properly infinite von Neumann algebra  \cite{AK2016, Ber2014}. 
But for type II\(_1\)   von
Neumann algebra \(M\) description of  structure of derivations on
the algebra \(S(M)\equiv LS(M)\) is still an open problem (see
\cite{AK2016}). In this connection  it should be noted that
Theorem~~\ref{localmeas} is  one of the first results on 2-local
derivations without information on  the general form  of
derivations on these algebras.
\end{remark}

\end{document}